\documentclass[12pt]{amsart}

\usepackage[dvips]{graphicx}
\usepackage{amsmath}
\usepackage{amsthm}
\usepackage{amsfonts}
\usepackage{amssymb}

\newtheorem{thm}{Theorem}[section]

\newtheorem{cor}[thm]{Corollary}
\newtheorem{lem}[thm]{Lemma}

\theoremstyle{definition}
\newtheorem{defn}[thm]{Definition}

\theoremstyle{remark}

\newcommand{\cycred}[1]{[#1]}
\newcommand{\proj}{\pi}
\newcommand{\vertices}{V}
\newcommand{\Star}[1]{{#1^\star}}
\newcommand{\gfamily}{\mathcal{G}}
\newcommand{\support}[1]{\operatorname{supp}(#1)}
\newcommand{\cyclicsupport}[1]{\operatorname{csupp}(#1)}
\newcommand{\starofCPword}[1]{#1^\star}
\newcommand{\equalsreduced}{\overset{red}{=}}

\newcommand{\co}{\colon\thinspace}
\newcommand{\determinant}{\operatorname{det}}

\begin{document}
\title[Rigidity of graph products]{Rigidity of graph products of abelian groups}
\address{Department of Mathematics, Tufts University, 503 Boston Ave, Medford 02155, USA}
\email{Mauricio.Gutierrez@tufts.edu, Adam.Piggott@tufts.edu}
\thanks{The authors wish to thank George McNinch for a number of helpful conversations during the preparation of this article.}
\date{\today}
\author{Mauricio Gutierrez and Adam Piggott}

\maketitle

\begin{abstract}
We show that if $G$ is a group and $G$ has a graph-product
decomposition with finitely-generated abelian vertex groups, then
$G$ has two canonical decompositions as a graph product of groups: a
unique decomposition in which each vertex group is a
directly-indecomposable cyclic group, and a unique decomposition in
which each vertex group is a finitely-generated abelian group and
the graph satisfies the $T_0$ property.  Our results build on
results by Droms, Laurence and Radcliffe.
\end{abstract}

\maketitle

\section{Introduction}

A \emph{labeled-graph} is a pair $(\Gamma, \gfamily_\Gamma)$, where
$\Gamma$ is a non-trivial finite simplicial graph with vertex set
$\vertices_\Gamma$ and $\gfamily_\Gamma = \{G_u\}_{u \in
\vertices_\Gamma}$ is a family of non-trivial groups (the
\emph{vertex groups}).  The graph product construction, first
defined in \cite{GreenThesis}, associates a group to each
labeled-graph: the graph product $W(\Gamma, \gfamily_\Gamma)$ is the
quotient of the free product $\ast_{u \in \vertices_\Gamma} G_u$ by
relations that allow elements of $G_u$ and $G_{u'}$ to commute if
$u$ and $u'$ are adjacent in $\Gamma$. The construction interpolates
between the free product construction, in the case that $\Gamma$ is
a discrete graph, and the direct product construction, in the case
that $\Gamma$ is a complete graph. We say that the labeled-graph
$(\Gamma, \gfamily_\Gamma)$ describes \emph{a graph-product
decomposition} of a group $G$ if $G \cong W(\Gamma,
\gfamily_\Gamma)$.

In the present article we study groups which have a graph-product
decomposition with finitely-generated abelian vertex groups (or
equivalently, cyclic vertex groups). A number of important classes
of groups have this property, including finitely-generated abelian
groups, finitely-generated non-abelian free groups, right-angled
Coxeter groups and right-angled Artin groups (also known as `graph
groups').

A \emph{labeled-graph isomorphism} $f\co  (\Gamma, \gfamily_\Gamma)
\to (\Sigma, \gfamily_\Sigma)$ is a bijection $f\co  V_\Gamma \to
V_\Sigma$ for which the following conditions hold:
\begin{enumerate}
\item $\forall u, v \in \vertices_\Gamma \;\; (u, v \hbox{ adjacent
in } \Gamma) \Leftrightarrow (f(u), f(v) \hbox{ adjacent in }
\Sigma)$;
\item $\forall u \in \vertices_\Gamma \;\; G_u \cong G_{f(u)}$.
\end{enumerate}
We write $(\Gamma, \gfamily_\Gamma) \cong (\Sigma, \gfamily_\Sigma)$
in case such a labeled-graph isomorphism exists.

\begin{figure}
\begin{center}
\includegraphics[scale=0.25]{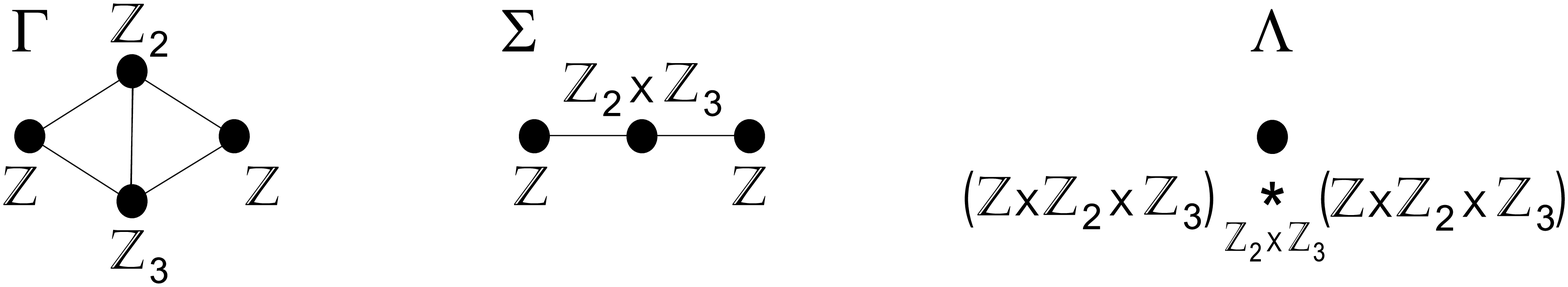}
\caption{$W(\Gamma, \gfamily_\Gamma) \cong W(\Sigma,
\gfamily_\Sigma) \cong W(\Lambda, \gfamily_\Lambda)$.
\label{RigidityFailureExample}}
\end{center}
\end{figure}

The example in Figure \ref{RigidityFailureExample} shows that
non-isomorphic labeled-graphs may determine isomorphic graph
products. If one wishes to study a class of groups $\mathfrak{G}$,
and one hopes to take advantage of graph-product decompositions when
doing so, it is desirable to identify a canonical isomorphism class
of labeled-graphs $(\Gamma, \gfamily_\Gamma)$ for each isomorphism
class of groups $G \in \mathfrak{G}$.  Further, the choice of
labeled-graph should be sufficiently `natural' that one may look to
the structure of $(\Gamma, \gfamily_\Gamma)$ to explain the
structure of $G$.  One way in which a graph-product decomposition
may be considered natural is if it is minimal, in some sense, and if
any two such minimal decompositions are isomorphic.


Droms \cite{Droms} proved that if $G$ has a graph-product
decomposition with infinite cyclic vertex groups, then any two such
decompositions are isomorphic. Using different methods, Laurence
\cite[Isomorphism Theorem for Graph Groups, p.329]{Laurence} proved
a stronger statement that includes information about a particular
labeled-graph isomorphism.  Radcliffe \cite{Radcliffe} proved that
if $G$ has a graph-product decomposition with
directly-indecomposable finite vertex groups, then any two such
decompositions are isomorphic.  Under the additional hypothesis that
the vertex groups are cyclic, a close reading of Radcliffe's
argument yields a result analogous to that of Laurence (Theorem
\ref{PreciseRigidityForGPIFG}).  None of the work mentioned above
allowed both finite and infinite vertex groups.

The main result of the present article is that if $G$ has a
graph-product decomposition with directly-indecomposable cyclic
vertex groups, then any two such decompositions are isomorphic
(Theorem \ref{RigidityForGPICG}). It follows that if $G$ has a
graph-product decomposition with finitely-generated abelian vertex
groups and a graph which satisfies the $T_0$ property (Definition
\ref{T0PropertyDefn}), then any two such decompositions are
isomorphic (Corollary \ref{RigidityForT0}).  Thus if $G$ is a group
and $G$ has a graph-product decomposition with finitely-generated
abelian vertex groups, then $G$ has two canonical decompositions as
a graph product of groups: a unique decomposition in which each
vertex group is a directly-indecomposable cyclic group (such as
$(\Gamma, \gfamily_\Gamma)$ in Figure \ref{RigidityFailureExample}),
and a unique decomposition in which each vertex group is a
finitely-generated abelian group and the graph satisfies the $T_0$
property (such as $(\Sigma, \gfamily_\Sigma)$ in Figure
\ref{RigidityFailureExample}).  The first decomposition is minimal
in the sense that the vertex groups are minimal, the second
decomposition is minimal in the sense that the graph has the least
number of vertices when we allow only finitely-generated abelian
groups as vertex groups.

Our techniques are combinatorial.  Our arguments make essential use
of the results of Droms, Laurence and Radcliffe mentioned above. In
$\S$\ref{PrelimSection} we discuss some preliminary results on
graphs and graph products, in $\S$\ref{T0Section} we remind the
reader of the $T_0$ property and the $T_0$-quotient of a
labeled-graph as used by Radcliffe, in
$\S$\ref{KnownRigidityResultsSection} we state a result by Laurence
and sharpen a result by Radcliffe and in
$\S$\ref{MainTheoremSection} we prove the main result.

\section{Graphs and graph products}\label{PrelimSection}


For a non-trivial finite simplicial graph $\Lambda$ (that is, a
graph with no circuits of length less than three), we write
$\vertices_\Lambda$ for the set of vertices of $\Lambda$. The
subgraph of $\Lambda$ determined by a subset $U \subseteq
\vertices_\Lambda$ is the full subgraph of $\Gamma$ determined by
the vertices in $U$. We write ${\rm MCS}(\Lambda)$ for the set of
maximal complete subgraphs (or cliques) of $\Lambda$.

Let $(\Lambda, \gfamily_\Lambda)$ be a labeled-graph. Each full
subgraph $\Theta$ of $\Lambda$ determines a labeled-subgraph
$(\Theta, \gfamily_{\Theta})$.  We write $W(\Theta)$ for the
subgroup of $W(\Lambda, \gfamily_\Lambda)$ generated by the natural
image of the set ${\cup}_{u \in \vertices_\Theta} G_u,$ and we note
that $W(\Theta) \cong W(\Theta, \gfamily_\Theta)$.  In particular,
we shall often write $W(\Lambda)$ for $W(\Lambda,
\gfamily_\Lambda)$.



A cyclic group is \emph{primary} if it has prime-power order, and
\emph{directly-indecomposable} if it has infinite order or it is
primary. For the remainder of this section we assume that
$\gfamily_\Lambda$ is a family of directly-indecomposable cyclic
groups. Following common practice, we abuse notation by ignoring the
formal distinction between the elements of $\vertices_\Lambda$ and
the generators of $W(\Lambda)$.

We write $\Lambda_T$ for the subgraph of $\Lambda$ determined by the
vertices $u \in \vertices_\Lambda$ for which $G_u$ has finite order
and we write $\Lambda_A$ for the subgraph of $\Lambda$ determined by
the vertices $u \in \vertices_\Lambda$ for which $G_u$ has infinite
order (we have selected $T$ for torsion and $A$ for Artin).  The
next lemma follows immediately from a more general result by Green
\cite[Theorem 3.26]{GreenThesis}.

\begin{lem}\label{MaximalFiniteSubgroupsAndMCS}
If $\gfamily_\Lambda$ is a family of directly-indecomposable cyclic
groups, then each maximal finite subgroup of $W(\Lambda)$ is abelian
and
$$\{W(\Theta) \; | \; \Theta \in {\rm MCS}(\Lambda_T)\}$$ is a complete set of representatives for the
conjugacy classes of maximal finite subgroups of $W(\Lambda)$.
\end{lem}

Following \cite{LaurenceThesis} \cite{Laurence}, a word $w$ in the
alphabet $\vertices_\Lambda^{\pm 1}$ is said to be \emph{reduced} if
there is no shorter word $w'$ which defines the same element of
$W(\Lambda)$. For a group element $g \in W(\Lambda)$ and a word $w$
in the alphabet $\vertices_\Lambda^{\pm 1}$, we write $g
\equalsreduced w$ if $w$ is a reduced word which equals $g$ in the
group $W(\Lambda)$. For words $w$ and $w'$ in the alphabet
$\vertices_\Lambda^{\pm 1}$, we write $w \equiv w'$ if $w$ and $w'$
are equal as words and we say that $w$ \emph{is transformed into
$w'$ by a letter swapping operation} if $w \equiv w_1 u^{\epsilon}
v^{\delta} w_2$ and $w' \equiv w_1 v^{\delta} u^{\epsilon} w_2$ for
some adjacent vertices $u, v \in \vertices_\Lambda$, some exponents
$\epsilon, \delta \in \{\pm 1\}$ and some reduced words $w_1, w_2$.
Laurence showed that for vertices $v_1, \dots, v_p \in
\vertices_\Lambda$ and integers $\epsilon_1, \dots, \epsilon_p \in
\{\pm 1\}$, if $w \equiv v_1^{\epsilon_1} \dots v_p^{\epsilon_p}$ is
not reduced, then there exists $1 \leq i < j \leq p$ such that $v_i
= v_j$, $\epsilon_i + \epsilon_j = 0$ and $v_i$ is adjacent to each
of the vertices $v_{i+1}, \dots, v_{j-1}$ (the \emph{Deletion
Condition}) \cite[Corollary 3.1.1]{LaurenceThesis}. It follows from
the Normal Form Theorem for Graph Products \cite{GreenThesis} (see
also \cite{LaurenceThesis}) that if two reduced words $w, w'$ define
the same element of $W(\Lambda)$, then $w$ can be transformed into
$w'$ by a finite number of letter swapping operations (the
\emph{Transpose Condition}). It follows from the Transpose Condition
that we may define
\begin{eqnarray*}
\support{g} & := & \{u \in \vertices_\Lambda \; | \; u \hbox{ or } u^{-1} \hbox{ appears in some reduced word for } g\} \\
\cyclicsupport{g} & := & \{u \in \vertices_\Lambda \; | \; \forall w
\in W(\Lambda) \;\; u \in \support{w g w^{-1}}\}.
\end{eqnarray*}
We say that $d$ is a \emph{terminal segment of $w$} if $d$ is a
reduced word and there exists a reduced word $w'$ such that $w
\equiv w'd$.

For an element $g \in W(\Lambda)$, we write $C_{W(\Lambda)}(g)$ for
the centralizer of $g$ in $W(\Lambda)$. An element $u \in
W(\Lambda)$ is said to be a CP element (for commuting product) if
there exists a complete subgraph $\Delta \subseteq \Gamma$ such that
$u \in W(\Delta)$. We write $\starofCPword{u}$ for the subgraph of
$\Gamma$ generated $\support{u}$ and those vertices adjacent to each
vertex in $\support{u}$. The centralizer of a CP element in
$W(\Lambda)$ has a particularly simple form.

\begin{lem}\label{CentralizerResult}
If $\gfamily_\Lambda$ is a family of directly-indecomposable cyclic
groups and $u$ is a CP element of $W(\Lambda)$, then
$C_{W(\Lambda)}(u) = W(\starofCPword{u}).$
\end{lem}

\begin{proof}
Let $u$ be a CP element of $W(\Lambda)$. We shall abuse notation by
also writing $u$ for a reduced word representing the group element
$u$. It is clear that $W(\starofCPword{u}) \subset C_{W(\Lambda)}(u)
$.

Suppose that $C_{W(\Lambda)}(u) - W(\starofCPword{u})$ is non-empty
and let $g$ be a minimal length (with respect to the generating set
$\vertices_\Gamma^{\pm 1}$) element of $C_{W(\Lambda)}(u) -
W(\starofCPword{u})$.  The minimality of $g$ implies that no
terminal segment of $g$ is contained in $W(\starofCPword{u})$, so $g
\equalsreduced w' y^\epsilon$ for some vertex $y \not \in \Star{u}$,
some exponent $\epsilon \in \{\pm 1\}$ and some reduced word $w'$.

We claim that $w' y^\epsilon u$ is a reduced word.  Suppose that $w'
y^\epsilon u$ is not a reduced word.  It follows from the Deletion
Condition that some sequence of letter swapping operations will
transform the word $w' y^\epsilon$ into a word $w'' v^\delta$, where
$v \in \support{u}$, $\delta \in \{\pm 1\}$ and $v^{-\delta}$
appears in $u$. But then $g$ has a terminal segment in
$W(\starofCPword{u})$.  This contradiction completes the proof of
the claim.

By hypothesis, $w' y^\epsilon u = u  w' y^\epsilon$ and it follows
that $u w' y^\epsilon$ is a also reduced word. By the transpose
condition, $w' y^\epsilon u$ may be transformed to $u w' y^\epsilon$
by a finite number of letter swapping operations.  Since $y^\epsilon
\not \in \Star{u}$, we have that $y \not \in \support{u}$.  It
follows that to transform $w' y^\epsilon u$ to $u w' y^\epsilon$ by
letter swapping operations, we must have that $y$ is adjacent to
each vertex in $\support{u}$.  But this contradicts the hypothesis
that $y \not \in \Star{u}$.
\end{proof}

\section{An equivalence relation on the vertices of a graph}\label{T0Section}

In this section we remind the reader of an equivalence relation on
the vertices of a graph which proved to be a key idea in
\cite{Radcliffe}.

Let $(\Lambda, \gfamily_\Lambda)$ be a labeled-graph. Following
\cite{Radcliffe}, we define a relation $\sim_\Lambda$ on the set
$\vertices_\Lambda$ as follows:
 $$u  \sim_\Lambda v \Leftrightarrow
(\forall \Theta \in {\rm MCS}(\Lambda) \;\;  u \in \Theta
\Leftrightarrow v \in \Theta).$$ It is easily verified that
$\sim_\Lambda$ is an equivalence relation.  We write $\tilde{u}$ for
the $\sim_\Lambda$ equivalence class of $u \in \vertices_\Lambda$.
It is immediate from the definitions that each $\sim_\Lambda$
equivalence class determines a complete subgraph of $\Lambda$.

\begin{defn}\label{T0PropertyDefn}
If each $\sim_\Lambda$ equivalence class is a singleton set, then we
say that $\Lambda$ \emph{satisfies the $T_0$ property}.
\end{defn}




The relation $\sim_\Lambda$ determines a well-defined quotient of
$\Lambda$ and a well-defined quotient of $(\Lambda,
\gfamily_\Lambda)$.

\begin{defn}\label{T0GraphDefn}
Let $\Lambda_0$ denote the graph with vertex set
$\vertices_{\Lambda_0}$ in one-to-one correspondence with the
$\sim_\Lambda$ equivalence classes of $\vertices_\Lambda$ and with
adjacency determined by the following rule:
$$\tilde{u} \hbox{ and } \tilde{v} \hbox{ adjacent in } \Lambda_0 \Leftrightarrow u \hbox{ and } v \hbox{ adjacent in } \Lambda \hbox{ and } u \not \sim_\Lambda v.$$
For each $\tilde{u} \in \vertices_{\Lambda_0}$, define
$G_{\tilde{u}} := \times_{v \in \tilde{u}} \; G_v$. Write
$\gfamily_{\Lambda_0} := \{G_{\tilde{u}} \}_{\tilde{u} \in
\vertices_{\Lambda_0}}.$ The graph $\Lambda_0$ is called the
\emph{$T_0$-quotient of} $\Lambda$ and the labeled-graph
$(\Lambda_0, \gfamily_{\Lambda_0})$ is called the
\emph{$T_0$-quotient of} $(\Lambda, \gfamily_\Lambda)$.
\end{defn}

We record some observations on the $T_0$-quotients $\Lambda_0$ and
$(\Lambda_0, \gfamily_{\Lambda_0})$

\begin{lem}\label{T0Lemma}
Let $(\Lambda, \gfamily_\Lambda)$ be a labeled-graph and let
$(\Lambda_0, \gfamily_{\Lambda_0})$ be its $T_0$-quotient. The
following properties hold:
\begin{enumerate}
\item \label{T0QuotientIsT0} $\Lambda_0$ satisfies the $T_0$ property;
\item \label{AdjacenyInT0} vertices $u, v \in \vertices_\Lambda$ are adjacent in $\Lambda$ if and
only if either $\tilde{u}$ and $\tilde{v}$ are adjacent in
$\Lambda_0$ or $\tilde{u} = \tilde{v}$;
\item \label{T0QuotientIsomorphic} $W(\Lambda, \gfamily_\Lambda) \cong W(\Lambda_0, \gfamily_{\Lambda_0})$.
\end{enumerate}
\end{lem}

We now establish that in the case that $\gfamily_\Lambda$ is a
family of directly-indecomposable cyclic groups, the isomorphism
class of $(\Lambda, \gfamily_{\Lambda})$ is uniquely determined by
the isomorphism class of the $T_0$-quotient $(\Lambda_0,
\gfamily_{\Lambda_0})$.

\begin{lem}\label{IsomorphismAndT0Isomorphism}
Let $(\Lambda, \gfamily_\Lambda)$ and $(\Xi, \gfamily_\Xi)$ be
labeled-graphs with directly-indecomposable cyclic vertex groups.
Then $(\Lambda, \gfamily_\Lambda)$ and $(\Xi, \gfamily_\Xi)$ are
isomorphic if and only if $(\Lambda_0, \gfamily_{\Lambda_0})$ and
$(\Xi_0, \gfamily_{\Xi_0})$ are isomorphic.
\end{lem}

\begin{proof}
It is clear that each isomorphism $(\Lambda, \gfamily_\Lambda) \to
(\Xi, \gfamily_\Xi)$ induces an isomorphism $(\Lambda_0,
\gfamily_{\Lambda_0}) \to (\Xi_0, \gfamily_{\Xi_0})$.  Now suppose
that there exists an isomorphism $t_0\co  (\Lambda_0,
\gfamily_{\Lambda_0}) \to (\Xi_0, \gfamily_{\Xi_0})$.  Each element
of $\gfamily_{\Lambda_0}$ (resp. $\gfamily_{\Xi_0}$) is a
finitely-generated abelian group. It is well-known that such groups
have a unique decomposition as a direct product of
directly-indecomposable cyclic groups.  Thus, for each $\tilde{u}
\in \vertices_{\Lambda_0}$, there exists a labeled-graph isomorphism
$t_{\tilde{u}}$ from the subgraph of $\Lambda$ generated by the
vertices in $\tilde{u}$ to the subgraph of $\Xi$ generated by the
vertices in $t_0(\tilde{u})$. Define $t$ to be the bijection
$\vertices_\Lambda \to \vertices_\Xi$ which restricts to
$t_{\tilde{u}}$ for each equivalence class $\tilde{u} \in
\vertices_{\Lambda_0}$. It follows from the definitions and Lemma
\ref{T0Lemma}(\ref{AdjacenyInT0}) that $t$ is a labeled-graph
isomorphism.
\end{proof}

\section{Some preliminary rigidity results}\label{KnownRigidityResultsSection}

For the remainder of this paper we assume that $(\Gamma,
\gfamily_\Gamma)$ and $(\Sigma, \gfamily_\Sigma)$ are labeled-graphs
with directly-indecomposable cyclic vertex groups.  As in
$\S$\ref{PrelimSection}, we shall abuse notation by ignoring the
formal distinction between the elements of $\vertices_\Gamma$ (resp.
$\vertices_\Sigma$) and the generators of $W(\Gamma)$ (resp.
$W(\Sigma)$).

As stated in the introduction, Droms \cite{Droms} proved that if $G$
has a graph-product decomposition with infinite cyclic vertex
groups, then any two such decompositions are isomorphic.  Laurence
proved the following stronger statement using peak reduction
techniques.

\begin{thm}[Laurence, p.329 \cite{Laurence}]\label{PreciseRigidityForRAAG}
If $\gfamily_\Gamma$ and $\gfamily_\Sigma$ are families of infinite
cyclic groups and $\alpha\co  W(\Gamma) \to W(\Sigma)$ is an
isomorphism, then there exists a labeled-graph isomorphism $a\co
(\Gamma, \gfamily_\Gamma) \to (\Sigma, \gfamily_\Sigma)$ for which
the following property holds:
$$\forall u \in \vertices_\Gamma \;\; a(u) \in \cyclicsupport{\alpha(u)}.$$
\end{thm}



Also, as stated in the introduction, Radcliffe \cite{Radcliffe}
proved that if $G$ has a graph-product decomposition with
directly-indecomposable finite vertex groups, then any two such
decompositions are isomorphic.  In this section we work towards a
full analogue of Theorem \ref{PreciseRigidityForRAAG}, under the
hypothesis of primary cyclic vertex groups.

By Lemma \ref{MaximalFiniteSubgroupsAndMCS}, if $\gfamily_\Lambda$
is a family of directly-indecomposable cyclic groups and $g \in
W(\Lambda)$ is an element of finite order, then there exists a
unique element $\cycred{g}$ of minimal length in the conjugacy class
of $g$ and there exists $\Theta \in {\rm MCS}(\Lambda_T)$ such that
$\cyclicsupport{g} \subseteq \Theta$ and $\cycred{g} \in W(\Theta)$.
This fact plays a central role in the present article because of the
following lemma.

\begin{lem}\label{SimplyfyingIsoOfT}
If $\gfamily_\Gamma$ and $\gfamily_\Sigma$ are families of primary
cyclic groups and $\tau\co  W(\Gamma) \to W(\Sigma)$ is an
isomorphism, then the map $\hat{\tau}\co \vertices_\Gamma \to
W(\Sigma)$ determined by the rule $u \mapsto [\tau(u)]$ extends to
an isomorphism $\hat{\tau}\co  W(\Gamma) \to W(\Sigma)$.
\end{lem}

\begin{proof}
Consider the presentation of $W(\Gamma)$ implicit in the graph
product decomposition $(\Gamma, \gfamily_\Gamma)$.  We shall show
that $\hat{\tau}$ extends to a homomorphism $W(\Gamma) \to
W(\Sigma)$ by checking that the relations in this presentation of
$W(\Gamma)$ are `preserved' by $\hat{\tau}$. Since $\tau$ is an
isomorphism and conjugation preserves the order of an element, it is
clear that the order of each vertex is preserved by $\hat{\tau}$. If
$u, v \in \vertices_\Gamma$ are adjacent in $\Gamma$, then $\langle
u, v \rangle$ has finite order and so does $\langle \tau(u), \tau(v)
\rangle$. By Lemma \ref{MaximalFiniteSubgroupsAndMCS}, there exists
$\Theta \in {\rm MCS}(\Sigma)$ and $w \in W(\Sigma)$ such that
$\tau(u), \tau(v) \in w W(\Theta) w^{-1}$. Then
$\hat{\tau}(u),\hat{\tau}(v)$ are contained in the abelian subgroup
$W(\Theta)$. Thus $\hat{\tau}(u) \hat{\tau}(v)\hat{\tau}(u)^{-1}
\hat{\tau}(v)^{-1} = 1$ and the relation $uvu^{-1} v^{-1}$ is
preserved by $\hat{\tau}$.

We shall show that the homomorphism $\hat{\tau}$ is an isomorphism
by showing that $\widehat{\tau^{-1}} \circ \hat{\tau}(v) = v$ for
each $v \in \vertices_\Gamma$.  Let $v \in \vertices_\Gamma$.  There
exists $\Theta \in {\rm MCS}(\Sigma)$, $a \in W(\Theta)$ and $w \in
W(\Sigma)$ such that $\tau(v) = w a w^{-1}$ and $\hat{\tau}(v) = a$.
Then $\tau^{-1}(a) = \tau^{-1}(w^{-1}) v \tau^{-1}(w)$ and
$\widehat{\tau^{-1}}(a) = v$.  So $\widehat{\tau^{-1}} \circ
\hat{\tau}(v) = \widehat{\tau^{-1}}(a) = v$, as required.
\end{proof}


We also require the following result concerning the isomorphisms of
finite abelian groups.

\begin{lem}\label{AbelianGroupLemma}
If $\Gamma$ and $\Sigma$ are complete graphs and $\gfamily_\Gamma$
and $\gfamily_\Sigma$ are families of primary cyclic groups and
$\tau\co W(\Gamma) \to W(\Sigma)$ is an isomorphism, then there
exists a labeled-graph isomorphism $t\co (\Gamma, \gfamily_\Gamma)
\to (\Sigma, \gfamily_\Sigma)$ for which the following property
holds:
$$\forall u \in \vertices_\Gamma \;\; t(u) \in \support{\tau(u)}.$$
\end{lem}

\begin{proof}
Without loss of generality we may assume that $\gfamily_\Gamma$ and
$\gfamily_\Sigma$ are families of $p$-primary cyclic groups for a
fixed prime $p$.

Let $\vertices_\Gamma = \{g_1, g_2, \dots, g_k\}$, let
$\vertices_\Sigma = \{s_1, s_2, \dots, s_k\}$ and let $A = (a_{\ell
m})$ be the matrix of integers such that $s_m$ appears with exponent
sum $a_{\ell m}$ in $\tau(g_\ell)$. Recall that
\begin{equation}\label{determinantdefn}
\determinant A = \sum_{\sigma \in {\rm Sym}(n)} sgn(\sigma) \, a_{1
\sigma(1)} a_{2 \sigma(2)} \dots a_{k \sigma(k)},
\end{equation}
where ${\rm Sym(n)}$ denotes the symmetric group on the set $\{1, 2,
\dots, n\}$. By \cite[Theorem 3.6]{AutOfAbelianGroups}, $A \; ({\rm
mod} \; p) \in {\rm GL}_k({\mathbb F}_{p})$.  Hence $\determinant A$
is not divisible by $p$ and at least one term of the sum
(\ref{determinantdefn}) is not divisible by $p$.  Thus there exists
$\sigma \in {\rm Sym}(n)$ such that $a_{m \, \sigma(m)}$ is
nontrivial modulo $p$ for each $m$; hence $s_{\sigma(m)} \in
\support{\tau(g_m)}$ for each $m$.   Define $t(g_m) = s_{\sigma(m)}$
for each $1 \leq m \leq k$.  By construction, $t$ has the required
properties.
\end{proof}

We now give an analogue of Theorem \ref{PreciseRigidityForRAAG}. The
proof below is an interpretation of Radcliffe's argument
\cite{Radcliffe}, with Lemma \ref{AbelianGroupLemma} applied at the
appropriate place to strengthen the result.

\begin{thm}[cf. Radcliffe \cite{Radcliffe}]\label{PreciseRigidityForGPIFG}
If $\gfamily_\Gamma$ and $\gfamily_\Sigma$ are families of primary
cyclic groups and $\tau\co  W(\Gamma) \to W(\Sigma)$ is an
isomorphism, then there exists a labeled-graph isomorphism $t\co
(\Gamma, \gfamily_\Gamma) \to (\Sigma, \gfamily_\Sigma)$ for which
the following property holds:
$$\forall u \in \vertices_\Gamma \;\; t(u) \in \cyclicsupport{\tau(u)}.$$
\end{thm}

\begin{proof}


In this paragraph we define a map $t_0\co \vertices_{\Gamma_0} \to
\vertices_{\Sigma_0}$.  Let $\tilde{u} \in \vertices_{\Gamma_0}$. It
follows from the $T_0$ property that $\tilde{u} \in
\vertices_{\Gamma_0}$ is uniquely identified by its memberships and
non-memberships in elements of ${\rm MCS}(\Gamma_0)$.  That is, the
singleton set $\{\tilde{u}\}$ is the intersection of the maximal
complete subgraphs of $\Gamma_0$ which contain $\tilde{u}$ minus the
union of the maximal complete subgraphs of $\Gamma_0$ which do not
contain $\tilde{u}$.  It follows from the definition that
$\hat{\tau}$ determines a one-to-one correspondence between the sets
${\rm MCS}(\Gamma_0)$ and ${\rm MCS}(\Sigma_0)$. Thus
$\hat{\tau}(G_{\tilde{u}})$ may be written as an intersection of
elements in ${\rm MCS}(\Sigma_0)$ minus a union of elements in ${\rm
MCS}(\Sigma_0)$. The $T_0$ property then implies that
$\hat{\tau}(G_{\tilde{u}}) = G_{\tilde{v}}$ for some $\tilde{v} \in
\vertices_{\Sigma_0}$. Define $t_0(\tilde{u}) = \tilde{v}$.

From the definitions (or using Lemma \ref{CentralizerResult}), the
reader may confirm that $t_0$ is a labeled-graph isomorphism and
that $\hat{\tau}$ restricts to an isomorphism $W(G_{\tilde{u}}) \to
W(G_{t_0(\tilde{u})})$ for each $\tilde{u} \in
\vertices_{\Gamma_0}$.  Following the proof of Lemma
\ref{IsomorphismAndT0Isomorphism}, we may lift $t_0$ to a
labeled-graph isomorphism $t\co (\Gamma, \gfamily_\Gamma) \to
(\Sigma, \gfamily_\Sigma)$.

Now, recall that each $\sim_\Gamma$ (resp. $\sim_\Sigma$)
equivalence class of vertices $\tilde{u}$ determines a complete
subgraph of $\Gamma$ (resp. $\Sigma$) and hence a finite abelian
subgroup of $W(\Gamma)$ (resp. $(W(\Sigma)$). By Lemma
\ref{AbelianGroupLemma}, we may choose the lift $t$ of $t_0$ so
that, on each subgraph of $\Gamma$ determined by a single
equivalence class $\tilde{u}$ of vertices, $t$ restricts to a
labeled-graph isomorphism with the property that $t(u) \in
\support{\hat{\tau}(u)}$ for each $u \in \tilde{u}$.  It follows
that $t(u) \in \cyclicsupport{\tau(u)}$ for each $u \in
\vertices_\Gamma$.
\end{proof}

\section{The Main Theorem}\label{MainTheoremSection}

We remind the reader that $(\Gamma, \gfamily_\Gamma)$ and $(\Sigma,
\gfamily_\Sigma)$ are labeled-graphs with directly-indecomposable
cyclic vertex groups.  We now assume that there exists a group
isomorphism $\phi\co  W(\Gamma) \to W(\Sigma)$.  Our task is to
exhibit a labeled-graph isomorphism $(\Gamma, \gfamily_\Gamma) \to
(\Sigma, \gfamily_\Sigma)$.

Let $T(\Gamma)$ (resp. $T(\Sigma)$) denote the subgroup of
$W(\Gamma)$ (resp. $W(\Sigma)$) generated by the elements of finite
order.
Let $\rho_{\Sigma_A}\co  W(\Sigma) \to W(\Sigma_A)$ denote the
retraction homomorphism determined by $$\forall u \in
\vertices_\Sigma \;\; u
\mapsto \left\{%
\begin{array}{ll}
  u & \hbox{if } u \in \vertices_{\Sigma_A} \\
  1 & \hbox{if } u \in \vertices_{\Sigma_T}. \\
\end{array}%
\right.$$




\begin{lem}\label{IsoRestrictsToA}
Let $\alpha\co   \vertices_{\Gamma_A} \to W(\Sigma_A)$ be defined by
$v \mapsto \rho_{\Sigma_A} \circ \phi(v).$ Then $\alpha$ extends to
an isomorphism $\alpha\co  W(\Gamma_A) \to W(\Sigma_A)$.
\end{lem}

\begin{proof}
Since $\phi(T(\Gamma)) = T(\Sigma)$, the isomorphism $\phi\co
W(\Gamma) \to W(\Sigma)$ induces an isomorphism $\hat{\phi}\co
W(\Gamma) / T(\Gamma) \to W(\Sigma) / T(\Sigma)$. Since $W(\Gamma_A)
\cap T(\Gamma) = \{1\}$ and $\vertices_{\Gamma_T} \subset
T(\Gamma)$, the quotient map $W(\Gamma) \to W(\Gamma) / T(\Gamma)$
restricts to an isomorphism $\proj_{\Gamma_A}\co W(\Gamma_A) \to
W(\Gamma) / T(\Gamma)$. Similarly, the quotient map $W(\Sigma) \to
W(\Sigma) / T(\Sigma)$ restricts to an isomorphism
$\proj_{\Sigma_A}\co W(\Sigma_A) \to W(\Sigma) / T(\Sigma).$ Thus
$(\proj_{\Sigma_A})^{-1} \circ \hat{\phi} \circ \proj_{\Gamma_A}$ is
an isomorphism $W(\Gamma_A) \to W(\Sigma_A)$. Calculation confirms
that $(\proj_{\Sigma_A})^{-1} \circ \hat{\phi} \circ
\proj_{\Gamma_A} = \rho_{\Sigma_A} \circ \phi.$

\end{proof}

Recall that, for an element $w$ of finite order in $W(\Sigma)$, we
write $[w]$ for the unique element of minimal length in the
conjugacy class of $w$.  As with Lemma \ref{SimplyfyingIsoOfT}, the
following lemma may be verified by elementary means.

\begin{lem}[cf. Lemma \ref{SimplyfyingIsoOfT}]\label{IsoRestrictsToT}
Let $\tau\co \vertices_{\Gamma_T} \to W(\Sigma_T)$ be defined by $u
\mapsto [\phi(u)].$ Then $\tau$ extends to an isomorphism $\tau\co
W(\Gamma_T) \to W(\Sigma_T)$.
\end{lem}

Lemmas \ref{IsoRestrictsToA} and \ref{IsoRestrictsToT} allow us to
use Theorems \ref{PreciseRigidityForRAAG} and
\ref{PreciseRigidityForGPIFG} to prove our main result.

\begin{thm}\label{RigidityForGPICG}
If $(\Gamma, \gfamily_\Gamma)$ and $(\Sigma, \gfamily_\Sigma)$ are
labeled-graphs with directly-indecomposable cyclic vertex groups and
there exists a group isomorphism $\phi\co W(\Gamma) \to W(\Sigma)$,
then there exists a labeled-graph isomorphism $f\co (\Gamma,
\gfamily_\Gamma) \to (\Sigma, \gfamily_\Sigma)$.
\end{thm}

\begin{proof}
By Lemma \ref{IsoRestrictsToA} and Theorem
\ref{PreciseRigidityForRAAG}, there exists a labeled-graph
isomorphism $a\co (\Gamma_A, \gfamily_{\Gamma_A}) \to (\Sigma_A,
\gfamily_{\Sigma_A})$ as in the statement of Theorem
\ref{PreciseRigidityForRAAG}.  By Lemma \ref{IsoRestrictsToT} and
Theorem \ref{PreciseRigidityForGPIFG}, there exists a labeled-graph
isomorphism $t\co (\Gamma_T, \gfamily_{\Gamma_T}) \to (\Sigma_T,
\gfamily_{\Sigma_T})$ as in the statement of Theorem
\ref{PreciseRigidityForGPIFG}. Define $f\co  \vertices_\Gamma \to
\vertices_\Sigma$ to be the bijection
$$u \mapsto \left\{%
\begin{array}{cc}
  a(u) & \hbox{if } u \in \vertices_{\Gamma_A}, \\
  t(u) & \hbox{if } u \in \vertices_{\Gamma_T}. \\
\end{array}%
\right.$$  We claim that $f$ is a labeled-graph isomorphism.

It is immediate from the definitions that $G_u \cong G_{f(u)}$ for
each $u \in \vertices_\Gamma$. It remains to show only that $f$
preserves the structure of $\Gamma$. Since we know $a$ and $t$ to be
labeled-graph isomorphisms, it remains to show only that $f$
preserves adjacency between vertices in $\vertices_{\Gamma_A}$ and
vertices in $\vertices_{\Gamma_T}$.

Let $u \in \vertices_{\Gamma_T}$ and $v \in \vertices_{\Gamma_A}$ be
adjacent in $\Gamma$. Since $u$ has finite order in $W(\Gamma)$,
$\phi(u)$ has finite order in $W(\Sigma)$.  By Lemma
\ref{MaximalFiniteSubgroupsAndMCS}, there exists an inner
automorphism $\iota$ of $W(\Sigma)$ such that $\iota \circ \phi(u)$
is a CP element of $W(\Sigma)$.  Write $x := \iota \circ \phi(u)$
and $y := \iota \circ \phi(v)$. Since $x$ and $y$ commute and $x$ is
a CP element of $W(\Sigma)$, we have by Lemma
\ref{CentralizerResult} that $y \in W(\starofCPword{x})$.  Hence
$\support{y} \subset \starofCPword{x}$.  But $f(u) = t(u) \in
\cyclicsupport{\phi(u)} = \support{x}$ and $f(v) = a(v) \in
\cyclicsupport{\phi(v)} = \cyclicsupport{y} \subset \support{y}$.
Hence $f(u)$ and $f(v)$ are adjacent in $\Sigma$.

It follows from the above paragraph that $\Sigma$ has at least as
many edges as $\Gamma$. Similarly, by considering $\phi^{-1}$ we may
show that $\Gamma$ has at least as many edges as $\Sigma$, and hence
the edges of $\Gamma$ and the edges of $\Sigma$ are in one-to-one
correspondence. It follows that if $u \in \vertices_{\Gamma_T}$ and
$v \in \vertices_{\Gamma_A}$ are not adjacent in $\Gamma$, then
$t(u)$ and $a(v)$ are not adjacent in $\Sigma$.

Thus $f$ preserves the structure of $\Gamma$ and $f$ is a
labeled-graph isomorphism.
\end{proof}

\begin{cor}\label{RigidityForT0}
If $(\Gamma, \gfamily_\Gamma)$ and $(\Sigma, \gfamily_\Sigma)$ are
labeled-graphs with directly-indecomposable cyclic vertex groups and
there exists a group isomorphism $\phi\co W(\Gamma) \to W(\Sigma)$,
then there exists a labeled-graph isomorphism $f_0\co (\Gamma_0,
\gfamily_{\Gamma_0}) \to (\Sigma_0, \gfamily_{\Sigma_0})$.
\end{cor}

\begin{proof}
Follows immediately from Theorem \ref{RigidityForGPICG} and Lemma
\ref{IsomorphismAndT0Isomorphism}.
\end{proof}

%
%
%
\bibliographystyle{amsplain}
\bibliography{RigidityOfGPICGBib}

\providecommand{\bysame}{\leavevmode\hbox to3em{\hrulefill}\thinspace}
\providecommand{\MR}{\relax\ifhmode\unskip\space\fi MR }
\providecommand{\MRhref}[2]{%
  \href{http://www.ams.org/mathscinet-getitem?mr=#1}{#2}
}
\providecommand{\href}[2]{#2}
\begin{thebibliography}{1}

\bibitem{Droms}
Carl Droms, \emph{Isomorphisms of graph groups}, Proc. Amer. Math. Soc.
  \textbf{100} (1987), no.~3, 407--408.

\bibitem{GreenThesis}
Elisabeth~R. Green, \emph{Graph products of groups}, Ph.D. thesis, The
  University of Leeds, 1990.

\bibitem{AutOfAbelianGroups}
Christopher~J. Hillar and Darren~L. Rhea, \emph{Automorphisms of finite abelian
  groups}, To appear in Amer. Math. Monthly.

\bibitem{LaurenceThesis}
Michael~R. Laurence, \emph{Automorphisms of graph products of groups}, Ph.D.
  thesis, Queen Mary College, University of London, 1993.

\bibitem{Laurence}
\bysame, \emph{A generating set for the automorphism group of a graph group},
  J. London Math. Soc. (2) \textbf{52} (1995), no.~2, 318--334.

\bibitem{Radcliffe}
David~G. Radcliffe, \emph{Rigidity of graph products of groups}, Algebr. Geom.
  Topol. \textbf{3} (2003), 1079--1088 (electronic).

\end{thebibliography}

\end{document}